\documentclass[oneside]{amsart}
\usepackage{epsfig,amsmath,latexsym,amssymb,amsfonts,amsthm,cite}
\usepackage{graphicx}
\usepackage{color,bm}
\usepackage[all]{xy}
\usepackage[T1]{fontenc}
\newtheorem{theorem}{Theorem}[section]
\newtheorem{lemma}[theorem]{Lemma}

\newtheorem{example}{Example}
\newtheorem{corollary}[theorem]{Corollary}
\newtheorem{main}{Theorem}

\newtheorem{main_cor}[main]{Corollary}

\def\T{\mathbb{T} } 

\def\R{\mathbb{R} } 
\def\Z{\mathbb{Z} } 
\def\nbd{neighborhood } 
\def\nbds{neighborhoods } 
\def\R{\mathbb{R} } 
 
\def\Sv{\mathop{\mathrm{Sing}}(v)} 
\def\Pv{\mathop{\mathrm{Per}}(v)} 
\def\Cv{\mathop{\mathrm{Cl}}(v)}

\def\-{\ominus} 
\def\+{\oplus} 
\def\0{\circ}

\usepackage{color,bm}

\title[Flows with time-reversal symmetric limit sets on surfaces]{Flows with time-reversal symmetric limit sets on surfaces}
\author{Tomoo Yokoyama}
\date{\today}
\address{Applied Mathematics and Physics Division, Gifu University, Yanagido 1-1, Gifu, 501-1193, Japan\\}
\email{tomoo@gifu-u.ac.jp}
\thanks{The author was partially supported by JSPS Grant Number 20K03583}
\makeatletter
\@namedef{subjclassname@2020}{%
  \textup{2020} Mathematics Subject Classification}
\makeatother



\begin{document}

\maketitle

\begin{abstract}
The Long-time behavior of orbits is one of the most fundamental properties in dynamical systems. Poincar\'e studied the Poisson stability, which satisfies a time-reversal symmetric condition, to capture the property of whether points return arbitrarily near the initial positions after a sufficiently long time. Birkhoff introduced and studied the concept of non-wandering points, which is one of the time-reversal symmetric conditions. Moreover, minimality and pointwise periodicity satisfy the time-reversal symmetric condition for limit sets. This paper characterizes flows with the time-reversal symmetric condition for limit sets, which refine the characterization of irrational or Denjoy flows by Athanassopoulos. Using the description, we construct flows on a sphere with Lakes of Wada attractors and with an arbitrarily large number of complementary domains, which are flow versions of such examples of spherical homeomorphisms constructed by Boro\'{n}ski, \v{C}in\v{c}, and Liu.
\end{abstract}
%149 語 (987 文字です)
%147 語 (985 文字です)

\section{Introduction}

The Long-time behavior of orbits is one of the most fundamental properties in dynamical systems. 
In \cite{poincare1890,poincare1899}, Poincar{\'e} studied the Poisson stability, which satisfies a time-reversal symmetric condition, to capture a property whether points return arbitrarily near the initial positions after a sufficiently long time. 
%Such a property is called Poisson stability. 
In \cite{birkhoff1927dynamical}, Birkhoff introduced and studied the concept of non-wandering points, which is one of time-reversal symmetric conditions, by introducing the concepts of $\omega$-limit set and $\alpha$-limit set of a point. 
Moreover, minimality and pointwise periodicity satisfy the time-reversal symmetric condition for limit sets. 
%Moreover, some dynamics properties (e.g. minimality, pointwise periodicity, non-wandering property) satisfy time-reversal symmetric conditions. 
%Cherry showed that the set of orbits contained in the closure of a non-closed recurrent orbit of a flow on a manifold contains uncountably many Poisson orbits \cite{cherry1937topological}. 
Athanassopoulos characterized a flow that is either irrational or Denjoy on a closed surface by using a  time-reversal symmetric condition for non-closed Poisson stable orbits \cite{athanassopoulos1992characterization}. 
In addition, the $\omega$-limit sets and $\alpha$-limit sets of points for discrete dynamics on surfaces can be wild in general. 
For instance, Lakes of Wada continua, whose existence is first shown by  Brouwer \cite[p.427]{brouwer1910analysis},  appear naturally as attractors in the study of such discrete dynamical systems on surfaces \cite{coudene2006pictures,Hubbard1995henon,kennedy1991basins,plykin1974sources}. 
Furthermore, there was the long-standing problem whether Lakes of Wada continua could arise in complex dynamics, which was the analogue of a question of Fatou \cite[pp.51--52]{fatou1920equations} concerning Fatou components of rational functions.
The existence is recently demonstrated by Mart{\'\i}-Pete, Rempe, and Waterman \cite{marti2021wandering}. 

In this paper, we show the existence of Lakes of Wada continua for continuous dynamics with the time-reversal symmetric condition for limit sets on compact surfaces, and characterize flows with the time-reversal symmetric condition for limit sets, which generalizes the characterization of irrational or Denjoy flows by Athanassopoulos.
%refine the characterization of irrational or Denjoy flows into one of the flows with time-reversal symmetric limit sets. 
To describe the statement, we introduce some concepts as follows. 
A point $x$ has a time-reversal symmetric limit set if $\alpha(x) = \omega(x)$. 
A subset (resp. flow) has time-reversal symmetric limit sets if any point in it (resp. in the whole space) has a time-reversal symmetric limit set. 
%A flow has time-reversal symmetric limit sets if each point has a time-reversal symmetric limit set. 
A non-recurrent orbit is connecting quasi-separatrix if both the $\omega$-limit and the $\alpha$-limit sets are contained in a boundary component of the singular point set.
A closed connected invariant subset is a non-trivial quasi-circuit if it is a boundary component of an open annulus, contains a non-recurrent point, and consists of non-recurrent points and singular points. 
A non-trivial quasi-circuit is a limit quasi-circuit if it is the $\omega$-limit or $\alpha$-limit set of a point outside of it. 
Then we have the following characterization of flows with time-reversal symmetric limit sets. 

\begin{main}\label{main:01}
A flow on a compact connected surface has time-reversal symmetric limit sets if and only if one of the following statements holds exclusively: 
\\
{\rm(1)} The flow is an irrational flow. 
\\
{\rm(2)} The flow is a Denjoy flow. 
\\
{\rm(3)} The surface is either a torus or a Klein bottle, and there is a limit cycle whose complement is a transverse annulus.
\\
{\rm(4)} Each recurrent orbit is closed, and each non-recurrent orbit is either a connecting quasi-separatrix with a time-reversal symmetric limit set or is contained in an invariant open transverse annulus with time-reversal symmetric limit sets which are limit quasi-circuits. 
%connected component of the complement of the union of  limit quasi-circuits is either an invariant subset consisting of closed orbits or an invariant open transverse annulus with time-reversal symmetric limit sets. 

In case {\rm(4)}, the periodic point set is open, and the $\omega$-limit and $\alpha$-limit sets of any non-recurrent point are either contained in the singular point set or a limit quasi-circuit. 
\end{main}

Notice that limit quasi-circuits for flows with time-reversal symmetric limit sets on compact connected surfaces can become like Lakes of Wada continua (see Example~\ref{ex:wada04} for details). 
Moreover, from these constructions, cutting closed transversals and collapsing the new boundary components, we can construct flows on a sphere with Lakes of Wada attractors and with an arbitrarily large number of complementary domains which are flow versions of such attractors of spherical homeomorphisms constructed by Boro\'{n}ski, \v{C}in\v{c}, and Liu \cite{boronski2020prime} and such an attractor of a transcendental entire function constructed by Mart{\'\i}-Pete, Rempe and Waterman \cite{marti2021wandering}. 
In other words, we have the following statement. 

\begin{main_cor}\label{main:02}
For any integer $N \geq 2$, there is a flow on a sphere with the Lakes of Wada attractor and with $N$ complementary domains. 
\end{main_cor}

The present paper consists of four sections.
In the next section, as preliminaries, we introduce fundamental concepts.
In \S 3, we topologically characterize flows with time-reversal symmetric limit sets. 
In the final section, some examples are illustrated to describe the time-reversal symmetric condition for limit sets. 

\section{Preliminaries}

\subsection{Topological notion}

A surface is a two-dimensional paracompact manifold with or without boundary. 

Recall that a one-dimensional cell complex $\gamma$ is essential in a compact surface $S$ if and only if $\gamma$ is not null homotopic in $S^*$, where $S^*$ is the resulting surface from $S$ by collapsing all boundary components into singletons. 
We call that a subset of a surface $S$ is essential if it is not null homotopic in the resulting surface $S^*$ from $S$ by collapsing all boundary components into singletons.
Here a singleton is a set consisting of a point.

\subsubsection{End completions}

Consider the direct system $\{K_\lambda\}$ of compact subsets of a topological space $X$ and inclusion maps such that the interiors of $K_\lambda$ cover $X$.  
There is a corresponding inverse system $\{ \pi_0( X - K_\lambda ) \}$, where $\pi_0(Y)$ denotes the set of connected components of a space $Y$. 
Then the set of ends of $X$ is defined to be the inverse limit of this inverse system. 
Notice that $X$ has one end $x_{\mathcal{U}}$ for each sequence $\mathcal{U} := (U_i)_{i \in \mathbb{Z}_{>0}}$ with $U_i \supseteq U_{i+1}$ such that $U_i$ is a connected component of $X - K_{\lambda_i}$ for some $\lambda_i$, where $A - B$ is used instead of $A \setminus B$ when $A \subseteq B$. 
Considering the disjoint union $X_{\mathrm{end}}$ of $X$ and  $\{ \pi_0( X - K_\lambda ) \}$ as a set, a subset $V$ of the union $X_{\mathrm{end}}$ is an open \nbd of an end $x_{\mathcal{U}}$ if there is some $i \in \mathbb{Z}_{>0}$ such that $U_i \subseteq V$. 
Then the resulting topological space $X_{\mathrm{end}}$ is called the end completion (or end compactification) of $X$, introduced by Freudenthal \cite{Freudenthal1931end} (cf. \cite[Definition~1.6]{raymond1960end}). 
Note that the end completion is not compact in general. 

Let $T$ be an open connected surface. 
From \cite[Theorem 3]{richards1963classification}, the open surface $T$ is homeomorphic to the resulting surface from a compact surface by removing a closed totally disconnected subset and so the end completion $T_{\mathrm{end}}$ of $T$ is a compact surface. 
%%
%We call that a loop in $T$ bounds an end $e \in T_{\mathrm{end}} - T$ if there is an open disk containing $e$ whose boundary is the loop. 
%A loop $\gamma$ in $T$ is contractible to an end $e \in T_{\mathrm{end}} - T$ if there is a sequences $\mathcal{U} = (U_i)_{i \in \mathbb{Z}_{>0}}$ defining $e$ and there is an open disk $D \subset T$ with $\gamma = \partial D$ with $\bigcup_{n\geq N} U_n \subset D$ for some $N \in \mathbb{Z}_{>0}$. 

\subsubsection{Metric completion}\label{sec:mc}

The metric completion $X_{\mathrm{cpl}}$ of a metric space $(X, d_X)$ is the set of equivalence classes of Cauchy sequences in $X$, where the equivalence relation $\sim$ is defined as $(x_n)_{n \geq 0} \sim (y_n)_{n \geq 0}$ if $\lim_{n \to \infty} d \left(x_n, y_n\right) = 0$. 
It is  endowed with a complete metric $d_{\hat X}$ given by $d_{\hat X} \left([(x_n)_{n\geq 0} ], [(y_n)_{n\geq 0} ]\right)= \lim _{n\to \infty}d (x_n , y_n)$, where $[(x_n)_{n\geq 0}]$ is the equivalent class of $(x_n)_{n\geq 0}$  (cf.~\cite{kelley1975general}).

Let $S$ be a compact surface, $\Gamma$ a closed subset of $S$. 
Define a distance $d$ on the complement $T := S - \Gamma$ as follows: 
$d(x,y)$ is the infimum of the lengths of paths from $x$ to $y$. 
The canonical quotient map $p_{\mathrm{cpl}} \colon T_{\mathrm{cpl}} \to S$ is locally injective and the restriction $p_{\mathrm{cpl}}|_{T}$ is an inclusion. 
%
%\subsubsection{Filling of an open subset in a surface}\label{sec:fill}
%
%Define the filling $\mathrm{Fill}_{S}(U) \subseteq S$ of for an open subset $U$ on a surface $S$ as follows: $p \in \mathrm{Fill}_{S}(U)$ if there is an open disk $D$ containing $p$ such that the boundary $\partial D \subset U$ is a simple closed curve and is null homotopic in $S$. 

\subsection{Notion of dynamical systems}
By a flow, we mean a continuous $\mathbb{R}$-action on a surface. 
Let $v \colon  \R \times S \to S$ be a flow on a compact surface $S$. 
Then $v_t := v(t, \cdot)$ is a homeomorphism on $S$. 
A subset of $S$ is invariant (or saturated) if it is a union of orbits. 
%The saturation $\mathrm{Sat}_v(A)$ of a subset $A \subseteq S$ is the union of orbits intersecting $A$. 
An invariant closed subset is minimal if it has no non-empty proper invariant closed subsets. 
For an invariant closed set $\gamma$, define the stable manifold $W^s(\gamma) := \{ y \in X \mid \omega(y) \subseteq  \gamma \}$ and the unstable manifold $W^u(\gamma) := \{ y \in X \mid \alpha(y) \subseteq  \gamma \}$.
A point $x$ of $S$ is singular if $x = v_t(x)$ for any $t \in \R$, is periodic if there is a positive number $T > 0$ such that $x = v_T(x)$ and  $x \neq v_t(x)$ for any $t \in (0, T)$, and is closed if it is either singular or periodic. 
Denote by $\mathop{\mathrm{Sing}}(v)$ (resp. $\mathop{\mathrm{Per}}(v)$, $\mathop{\mathrm{Cl}}(v)$) the set of singular (resp. periodic, closed) points. 
%
%A point is wandering if there are its neighborhood $U$ and a positive number $N$ such that $v_t(U) \cap U = \emptyset$ for any $t > N$. Then such a neighborhood is called a wandering domain. 
%A point is non-wandering if it is not wandering (i.e. for any its neighborhood $U$ and for any positive number $N$, there is a number $t \in \mathbb{R}$ with $|t| > N$ such that $v_t(U) \cap U \neq \emptyset$).
%A point $x$ is non-wandering if for each neighborhood $U$ of $x$ and each positive number $N$, there is $t \in \mathbb{R}$ with $|t| > N$ such that $v_t(U) \cap U \neq \emptyset$. 
%An orbit is singular (resp. periodic, non-wandering) if it consists of singular (resp. periodic, non-wandering) points. 
For a point $x \in S$, define the $\omega$-limit set $\omega(x)$ and the $\alpha$-limit set $\alpha(x)$ of $x$ as follows: $\omega(x) := \bigcap_{n\in \mathbb{R}}\overline{\{v_t(x) \mid t > n\}} $, $\alpha(x) := \bigcap_{n\in \mathbb{R}}\overline{\{v_t(x) \mid t < n\}} $. 
%For an orbit $O$, define $\omega(O) := \omega(x)$ and $\alpha(O) := \alpha(x)$ for some point $x \in O$.
%Note that an $\omega$-limit (resp. $\alpha$-limit) set of an orbit is independent of the choice of a point in the orbit. 
%
A point $x$ of $S$ is Poisson stable (or strongly recurrent) if $x \in \omega(x) \cap \alpha(x)$. 
A point $x$ of $S$ is recurrent if $x \in \omega(x) \cup \alpha(x)$. 
Denote by $\mathrm{P}(v)$ (resp. $\mathrm{R}(v)$) the set of non-recurrent (resp. non-closed recurrent) points. 
Then $S = \mathop{\mathrm{Cl}}(v) \sqcup \mathrm{P}(v) \sqcup \mathrm{R}(v)$. 
The closure of a non-closed recurrent orbit is called a Q-set (or quasi-minimal set). 
%A quasi-minimal set is an orbit closure of a non-closed recurrent orbit. 
An orbit is singular (resp. periodic, closed, recurrent, Poisson stable) if it consists of singular (resp. periodic, closed, recurrent, Poisson stable) points. 
A periodic orbit is a limit cycle if it is the $\omega$-limit or $\alpha$-limit set of a point outside of it. 
A flow is Poisson stable (resp. recurrent) if each point is Poisson stable (resp. recurrent). 
%A flow is pointwise almost periodic if any orbit closures are minimal sets. 
%Notice that a flow is pointwise almost periodic if and only if the orbit class space is $T_1$. 

\subsubsection{Transversality and closed transversals}
Gutierrez's smoothing theorem~\cite{gutierrez1986smoothing} says that each flow on a compact surface is topologically equivalent to a $C^1$-flow. 
Therefore we can define transversality using tangential spaces of surfaces via topological equivalence. 
However, to modify transverse arcs, we immediately define transversality as follows.  
An open arc $C$ is transverse to $v$ if there is a small neighborhood $U$ of $C$ and there ist a homeomorphism $h:U \to [-1,1] \times (-1,1)$ with $h(C) = \{0\} \times (-1,1)$ such that $h^{-1}([-1,1] \times \{t \})$ for any $t \in (-1, 1)$ is an orbit arc and $h^{-1}(\{0\} \times [-1,1]) = C \cap U$. 
A simple closed curve is a closed transversal if it is a union of open arcs which are transverse to $v$. 
Notice that the concept of a closed transversal comes from the foliation theory (cf. \cite[Definition~3.4.1 p.41]{Hector1983foliation_A}). 

\subsubsection{Quasi-Q-set}
An invariant subset is a quasi-Q-set if it is an $\omega$-limit or $\alpha$-limit set of a point and intersects an essential closed transversal infinitely many times.

\subsubsection{Flow boxes, periodic annuli, and transverse annuli}
A closed trivial flow box is homeomorphic to $[0,1]^2$ each of whose orbit arcs correspond to $[0,1] \times \{t \}$ for some $t \in [0,1]$ as on the left of Figure~\ref{fig:local}. 
An open trivial flow box is homeomorphic to $(0,1)^2$ each of whose orbit arcs correspond to $(0,1) \times \{t \}$ for some $t \in (0,1)$. 
\begin{figure}
\begin{center}
\includegraphics[scale=0.35]{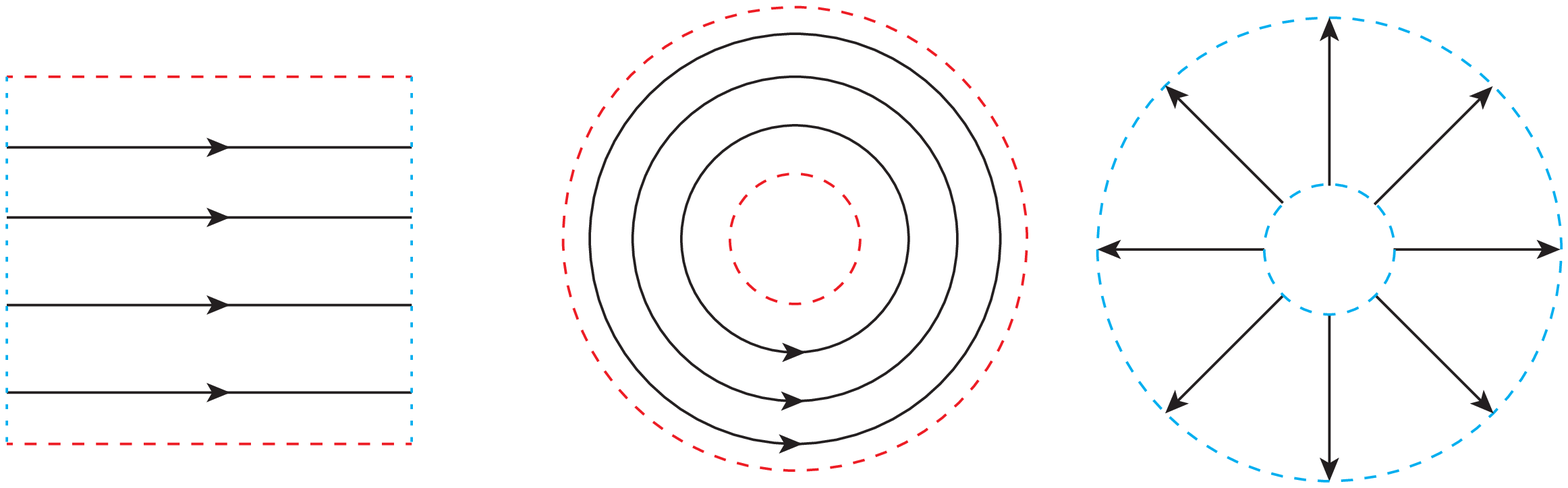}
\end{center}
\caption{Left, flow box; middle, periodic annulus; right, transverse annulus.}
\label{fig:local}
\end{figure}
A subset $U$ is an open (resp. closed) periodic annulus if there is an interval $I := (0,1)$ (resp. $I := [0,1]$) and there is a homeomorphism $h \colon U \to I \times \mathbb{S}^1$ whose images of any orbits in $U$ are maximal orbits $\{ t \} \times \mathbb{S}^1$ for some $t \in I$ as on the middle of Figure~\ref{fig:local}. 
A subset $U$ is an open (resp. closed) transverse annulus if there is an interval $I := (0,1)$ (resp. $I := [0,1]$) and there is a homeomorphism $h \colon U \to I \times \mathbb{S}^1$ whose images of any maximal orbit arcs in $U$ are maximal orbit arcs $I \times \{ \theta \}$ for some $\theta \in \mathbb{S}^1$ as on the right of Figure~\ref{fig:local}. 

%\subsubsection{Topological properties of orbits}
%
%An orbit is proper if it is embedded, locally dense if its closure has a nonempty interior, and exceptional if it is neither proper nor locally dense. 
%A point is proper (resp. locally dense, exceptional) if its orbit is proper (resp. locally dense, exceptional).
%Denote by $\mathrm{LD}(v)$ (resp. $\mathrm{E}(v)$, $\mathrm{P}(v)$) the union of locally dense orbits (resp. exceptional orbits, non-closed proper orbits). 
%Then $S = \mathop{\mathrm{Cl}}(v) \sqcup \mathrm{P}(v) \sqcup \mathrm{LD}(v) \sqcup \mathrm{E}(v)$, where $\sqcup$ denotes a disjoint union. 
%Note that an orbit on a paracompact manifold (e.g. a surface) is proper if and only if it has a neighborhood in which the orbit is closed \cite{yokoyama2019properness}.  
%This implies that a non-recurrent point is proper and so that a non-proper point is recurrent.  
%In \cite[Theorem~VI]{cherry1937topological}, Cherry showed that the closure of a non-closed recurrent orbit $O$ of a flow on a manifold contains uncountably many non-closed Poisson stable orbits whose closures are $\overline{O}$. 
%This means that each non-closed recurrent orbit of a flow on a manifold has no neighborhood in which the orbit is closed, and so is non-proper. 
%In particular, a non-closed proper orbit is non-recurrent. 
%Therefore the union $\mathrm{P}(v)$ of non-closed proper orbits is the set of non-recurrent points and that $\mathrm{R}(v) = \mathrm{LD}(v) \sqcup \mathrm{E}(v)$.
%Hence we have a decomposition $S = \mathop{\mathrm{Cl}}(v) \sqcup \mathrm{P}(v) \sqcup \mathrm{R}(v)$. 

\subsubsection{Time-reversal symmetric conditions for subsets}

%For a flow $v \colon \R \times S \to S$, the flow $-v$ defined by $-v(t,x) := v(-t,x)$ is called the time reversed flow of $v$. 
%A property $P$ for a subset $F$ of the flow $v$ is time-reversal symmetric if $F$ has the property $P$ with respect to the time reversed flow $-v$. 
%By definition, non-wandering property and Poisson stability are time-reversal symmetric. 
A point $x$ has a time-reversal symmetric limit set if $\alpha(x) = \omega(x)$. 
A subset (resp. flow) has time-reversal symmetric limit sets if any point in it (resp. in the whole space) has a time-reversal symmetric limit set. 
%A flow has time-reversal symmetric limit sets if each point has a time-reversal symmetric limit set. 
%A flow has almost time-reversal symmetric limit sets if there is an open dense subset which has the time-reversal symmetric limit set. 

\section{Characterization of flows with time-reversal symmetric limit sets}

%Let $v$ be a flow on a compact connected surface $S$. 
%\cite[Theorem~1.1]{athanassopoulos1992characterization} implies
We have the following observation. 

\begin{lemma}\label{lem:01}
The following statements are equivalent for a flow on a compact connected surface:
\\
{\rm(1)} The flow is either an irrational flow or a Denjoy flow. 
\\
{\rm(2)} There is a Poisson stable orbit and there is an invariant open \nbd of it with time-reversal symmetric limit sets. 
\\
{\rm(3)} The flow has a non-closed recurrent point and time-reversal symmetric limit sets. 
\end{lemma}

\begin{proof}
\cite[Theorem~1.1]{athanassopoulos1992characterization} implies that assertions {\rm(1)} and {\rm(2)} are equivalent.
Assertion {\rm(1)} implies assertion {\rm(3)}. 
Let $v$ be a flow on a compact connected surface $S$. 
Suppose that $v$ has a non-closed recurrent point $x$ and time-reversal symmetric limit sets. 
By \cite[Theorem~VI]{cherry1937topological}, the closure of the non-closed recurrent orbit $O(x)$  contains uncountably many Poisson orbits whose closures are $\overline{O}$. 
Therefore, assertion {\rm(2)} holds. 
\end{proof}

We obtain the following statement. 

\begin{lemma}\label{lem:01}
The following statements hold for a flow without non-closed recurrent points on a compact connected surface $S$: 
\\
{\rm(1)} If there are a limit cycle and its invariant open \nbd with time-reversal symmetric limit sets, then the surface is either a torus or a Klein bottle and the complement of the limit cycle is a transverse annulus. 
\\
{\rm(2)} If there is a non-recurrent point $x$ whose $\omega(x)$-limit set is a limit quasi-circuit and there is an invariant open \nbd of $\omega(x)$ with time-reversal symmetric limit sets, then the connected component of $W^u(\omega(x)) - \omega(x) = W^s(\omega(x)) - \omega(x)$ containing $x$ is an invariant open transverse annulus. 
%If there are a non-recurrent point $x$ whose $\omega(x)$-limit set is a limit quasi-circuit and the invariant open \nbd $V$ of $\omega(x)$ such that the set difference $V - \omega(x)$ has the time-reversal symmetric limit sets, then the connected component of $W^u(\omega(x)) - \omega(x) = W^s(\omega(x)) - \omega(x)$ containing $x$ is an invariant open transverse annulus. 
\end{lemma}

\begin{proof}
Let $\gamma$ be a limit cycle or  limit quasi-circuit, $U$ be an invariant open \nbd of $\gamma$, and $x$ be a point with $\gamma = \omega(x)$. 
Then $x \in U$ and $\gamma = \alpha(x) = \omega(x)$. 
Fix any non-recurrent point $y$ in $\gamma$ and a small open transverse arc $I$ intersecting $\{ y \} = I \cap \gamma$. 
\cite[Theorem~A]{yokoyama2021poincare} implies that if $\gamma$ is not periodic then $\alpha(y) = \omega(y) \subseteq \Sv$.  
Let $I_-$ and $I_+$ be the connected components of $I - \{ y \}$ with $O^+(x) \cap I \subset I_+$ and $O^-(x) \cap I \subset I_-$. 
Taking $I$ small if necessary, we may assume that $O^+(x) \cap I_+$ converges to $y$ monotonically in $I$ and that $O^-(x) \cap I_-$ converges to $y$ monotonically in $I$.  
For any $i \in \Z_{>0}$, denote by $x_i$ the $i$-th intersection $O^+(x) \cap I_+$ for $v$ and by $x_{-i}$ the $i$-th intersection $O^-(x) \cap I_-$ for $-v$. 
Let $f_v: I \to I$ be the first return map on $I$ induced by $v$, $C_{a,b} \subset O(x)$ the orbit arc from $a$ to $b$, and $I_{a,b} \subset I$ the subinterval between $a$ and $b$ of $I$.  
Then $x_{i+1} = f_v(x_i)$ and $x_{-i-1} = f_v(x_{-i})$ for any $i \in \Z_{>0}$. 
For any $i \in \Z_{>0}$, put $C_i := C_{x_i, x_{i+1}}$, $J_i := I_{x_i, x_{i+1}}$, $C_{-i} := C_{x_{-i}, x_{-i-1}}$, $J_{-i} := I_{x_{-i}, x_{-i-1}}$. 
For any $i \in \Z_{>0}$, define loops $\mu_i := C_i \cup J_i$ and $\mu_{-i} := C_{-i} \cup J_{-i}$. 
Then $\mu_i \cap \mu_j = \emptyset$ for any $i, j \in \Z - \{ 0 \}$ with $|i - j| > 1$. 
Moreover, we have $\mu_i \cap \mu_{i+1} = \{ x_{i+1} \}$ and $\mu_{-i} \cap \mu_{-(i+1)} = \{ x_{-(i+1)} \}$ for any $i \in \Z_{>0}$. 
%Then the unions $\mu_i := C_i \cup J_i$ and $\mu_{-i} := C_{-i} \cup J_{-i}$ for any $i \in \Z_{>0}$ are pairwise disjoint loops. 
By compactness of $S$, by taking a subarc of $I$, we may assume that the connected component of $S - (\mu_i \sqcup \mu_{i+4})$ (resp. $S - (\mu_{-i} \sqcup \mu_{-i-4})$) containing $\mu_{i+2}$ (resp. $\mu_{-i-2}$) for any $i \in \Z_{>0}$ is an open annulus. 
Applying \cite[Lemma~3.1]{yokoyama2021poincare} to the loops $\mu_i := C_i \cup J_i$ and $\mu_{-i} := C_{-i} \cup J_{-i}$ for any $i \in \Z_{>0}$, there are closed transversals $\gamma_i$ and $\gamma_{-i}$ near $\mu_i$ and $\mu_{-i}$ respectively. 
By construction, the connected component of $S - (\gamma_i \sqcup \gamma_{i+4})$ (resp. $S - (\gamma_{-i} \sqcup \gamma_{-i-4})$) containing $\gamma_{i+2}$ (resp. $\gamma_{-i-2}$) for any $i \in \Z_{>0}$ is an open transverse annulus $A_i$ (resp. $A_{-i}$). 
Then the union $A_+ := \bigsqcup_{i =0}^\infty A_{1+4i} \sqcup \gamma_{1+4i}$ is a positive invariant open transverse annulus and the union $A_- := \bigsqcup_{i =1}^\infty A_{-(1+4i)} \sqcup \gamma_{-(1+4i)}$ is a negative invariant open transverse annulus. 
By construction, we have that $\gamma = \omega(x) = \omega(z_+)$ for any $z_+ \in A_+$ and $\gamma = \alpha(x) = \alpha(z_-)$ for any $z_- \in A_-$. 
Moreover, the union $A_- \sqcup \gamma \sqcup A_+$ is a \nbd of $\gamma \setminus \Sv = \gamma \cap (\Pv \sqcup \mathrm{P}(v))$ and is contained in any invariant open \nbd of $\gamma$. 
Then $A_- \sqcup \gamma \sqcup A_+$ has time-reversal symmetric limit sets. 
Since the $\omega$-limit and $\alpha$-limit sets of every point in $A_- \sqcup \gamma \sqcup A_+$ are time-reversal symmetric, the intersection $O^-(z_+) \cap A_-$ for any $z_+ \in \gamma_1$ is negative invariant and the intersection $O^+(z_-) \cap A_+$ for any $z_- \in \gamma_{-1}$ is positive invariant. 
For any $z_+ \in \gamma_1$, denote by $C_{z_+}$ the closed orbit arc from $z_+$ to the intersection $\gamma_{-1} \cap O^-(z_+)$. 
By the flow box theorem for a continuous flow on a compact surface (cf. Theorem 1.1, p.45\cite{aranson1996introduction}), for any $z_+ \in \gamma_1$, there is a closed flow box $B_{z_+}$ containing $C_{z_+}$ such that $\bigsqcup_{z_+' \in \gamma_1 \cap B_{z_+}} C_{z_+'} = B_{z_+}$.  
From the compactness of $\gamma_1$, there are finitely many closed flow boxes $B_{z_{+,i}}$ whose union is a cover of $\gamma_1$. 
Then the union $A_0 := \bigcup_i B_{z_{+,i}}$ is a closed transverse annulus whose boundary is $\gamma_1 \sqcup \gamma_{-1}$. 
Therefore the disjoint union $A := A_- \sqcup A_0 \sqcup A_+$ is an invariant open annulus with $x \in A$. 
Since the union $A \sqcup \gamma$ is a \nbd of $\gamma \setminus \Sv = \gamma \cap (\Pv \sqcup \mathrm{P}(v))$, 
%the disjoint union $A \sqcup (\gamma \cap (\Pv \sqcup \mathrm{P}(v))) = (A \sqcup \gamma) \setminus \Sv$ is an invariant subset. 
%This implies that 
the connected component of $W^u(\gamma) - \gamma = W^s(\gamma) - \gamma$ containing $x$ is the invariant open transverse annulus $A$. 

Suppose that $\gamma$ is a limit cycle. 
Then the invariant open subset $A \sqcup \gamma$ is a \nbd of $\gamma$ with $\gamma = \partial A$. 
Therefore $A \sqcup \gamma$ is closed and open. 
Since $S$ is connected, we have that $S = U = A \sqcup \gamma$ is the resulting surface from a closed annulus gluing the boundary and is either a torus or a Klein bottle. 
\end{proof}

Notice that, for any positive integer $k$, there exists a flow with time-reversal symmetric limit sets, but without non-closed recurrent points on a compact connected surface $S$, for which there is a limit quasi-circuit $\gamma$ such that the complement $S - \gamma$ consists of $k$ connected components. 
%there exists a flow with time-reversal symmetric limit sets, but without non-closed recurrent points on a compact connected surface, for which there is an $\omega$-limit set $\gamma$ of a point such that the complement $W^u(\gamma) - \gamma$ consists of $k$ connected components. 
%
Indeed, by a Wada-Lakes-like construction (see Example~\ref{ex:wada04} for details), we can construct such a flow with a limit quasi-circuit whose complement consists of invariant subsets in $\Cv$ and invariant open transverse annuli. 
The previous lemma implies the following statement. 

\begin{corollary}\label{cor:01}
The following statements are equivalent for a flow on a compact connected surface:
\\
{\rm(1)} The surface is either a torus or a Klein bottle, and there is a limit cycle whose complement is a transverse annulus. 
\\
{\rm(2)} There is a limit cycle, and there is an invariant open \nbd of it with time-reversal symmetric limit sets. 
\\
{\rm(3)} The flow has a limit cycle and time-reversal symmetric limit sets. 
\end{corollary}

\begin{proof}
Lemma~\ref{lem:01} implies that assertions {\rm(1)} and {\rm(2)} are equivalent. 
Assertion {\rm(1)} implies assertion {\rm(3)}. 
Assertion {\rm(3)} implies assertion {\rm(2)}. 
\end{proof}

For a closed subset of a compact surface, we have the following relation between its connected components and the ends of its complement. 
%Roughly speaking, the compact metrizable space $T'$ can be obtained from $T_{\mathrm{end}}$ by identifying ends along with connected components of $\Sv \cap Q$. 
%More precisely, we have the following statement. 
%To demonstrate this, we have the following claims. 

\begin{lemma}\label{lem:end_collapse}
Let $S$ be a connected compact surface and $T$ a connected component of the complement $S -A$ of a closed subset $A$ of $S$. 
% containing a point $x \in S - A$. 
Denote by $T'$ the resulting space $T'$ from $T \sqcup A$ by collapsing any connected component of $A$ into singletons. 
Then there is a continuous extension $p \colon T_{\mathrm{end}} \to T'$ of the inclusion $T \subset T'$ such that any inverse images of singletons on $T'$ by $p$ are finite.  
\end{lemma}

\begin{figure}
\[
\xymatrix@=18pt{
& & T_{\mathrm{cpl}} \ar[lld]_{p_{\mathrm{cpl}}} \ar[rrd]^{q_{\mathrm{cpl}}} & & \\
T \sqcup (\Sv \cap Q) \ar@{}[d]|{\bigcup} \ar[rr]^{\hspace{30pt}p_0} & & T'  \ar@{}[d]|{\bigcup}  & &  T_{\mathrm{end}} \ar[ll]_p \ar@{}[d]|{\bigcup} \\
T & &  T \ar@{=}[ll] \ar@{=}[rr]  & &  T}
\]
\caption{A commutative diagram induced by the end completion, the metric completion and the collapse.}
\label{Fig:quotients}
\end{figure}

\begin{proof}
For a connected component $C$ of $A$ intersecting $\partial T$, define a subset $\mathcal{E}_C(T)$ of $\mathcal{E}(T)$ by $e \in \mathcal{E}_C(T)$ if $e \in \overline{U_C - C}^{T_{\mathrm{end}}}$ for any \nbd $U_C \subseteq T$ of $C$, where $\overline{B}^{T_{\mathrm{end}}}$ is the closure of $B$ with respect to $T_{\mathrm{end}}$.

We claim that $\mathcal{E}_C(T)$ is finite for any connected component $C$ of $A$ intersecting $\partial T$. 
Indeed, assume that $\mathcal{E}_C(T)$ is infinite for some connected component $C$ of $A$ intersecting $\partial T$. 
Let $e_i$ be infinitely many ends of $\mathcal{E}_C(T)$. 
Then there are infinitely many loops $\gamma_i \subset  T$ and there are infinitely many open disks $D_i \subset  T_{\mathrm{end}}$ with $e_i \in D_i$, $\partial D_i = \gamma_i$, $\mathcal{E}_C(T) \not\subseteq \bigcup_{i} \overline{D_i}^{T_{\mathrm{end}}}$, and $\overline{D_i}^{T_{\mathrm{end}}}\cap \overline{D_j}^{T_{\mathrm{end}}} = \emptyset$ for any $i \neq j$. 
%, where $\overline{B}^{T_{\mathrm{end}}}$ is the closure of $B$ with respect to $T_{\mathrm{end}}$. 
%
Since $T_{\mathrm{end}}$ is a connected compact surface and $D_i \subset T_{\mathrm{end}}$ is an open disk for any $i$, the set difference $T_{\mathrm{end}} - \bigcup_{i} \overline{D_i}^{T_{\mathrm{end}}}$ is a connected surface. 
By the openness and connectivity of the surface $T_{\mathrm{end}} - \bigcup_{i} \overline{D_i}^{T_{\mathrm{end}}}$, \cite[Theorem~1.5(f)]{raymond1960end} implies that $T \setminus \bigcup_{i} \overline{D_i}^{T_{\mathrm{end}}}$ is a connected subset of $S$. 
By $\mathcal{E}_C(T) \not\subseteq \bigcup_{i} \overline{D_i}^{T_{\mathrm{end}}}$, the closure $\overline{T \setminus \bigcup_{i} \overline{D_i}^{T_{\mathrm{end}}}} \subseteq C \sqcup (T \setminus \bigcup_{i} \overline{D_i}^{T_{\mathrm{end}}})$ intersects $C$. 
The connectivity of $C$ implies that the union $(T \setminus \bigcup_{i} \overline{D_i}^{T_{\mathrm{end}}}) \sqcup C$ is connected. 
On the other hand, since $\overline{T \cap D_i} \cap C \neq \emptyset$ for any $i$, the union $C \sqcup \bigcup_{i} (T \cap D_i) \subset S$ is connected. 
Then the union $((T \setminus \bigcup_{i} \overline{D_i}^{T_{\mathrm{end}}}) \sqcup C) \cup (C \sqcup \bigcup_{i} (T \cap D_i)) = (T - \bigcup_{i} \gamma_i) \sqcup C$ is connected. 
Since $T$ is a connected component of $S - C$, the connectivity of $S$ implies that $S - T$ is connected.  
Therefore the union $((T - \bigcup_{i} \gamma_i) \sqcup C) \cup (S - T) = S - \bigcup_{i} \gamma_i$ is connected, which contradicts the compactness of $S$.  

Denote by $T_{\mathrm{cpl}}$ the metric completion of $T$ and by $p_{\mathrm{cpl}} \colon T_{\mathrm{cpl}} \to \overline{T} \subseteq T \sqcup A$ the canonical quotient map as in \S~\ref{sec:mc}. 
We claim that an extension $q_{\mathrm{cpl}} \colon T_{\mathrm{cpl}} \to T_{\mathrm{end}}$ of the inclusion $T \subset T'$ is well-defined and continuous.  
Indeed, fix a point $\widetilde{x} \in T_{\mathrm{cpl}} - T$. 
By definition of metric completion, the point $\widetilde{x}$ is a Cauchy sequence $(x_n)_{n \in \Z_{\geq 0}}$ in $T$ and so a sequence converging to a point $x_\infty \in A \subseteq S - T$. 
Since $T_{\mathrm{cpl}}$ is a complete metric space defined by a Riemannian metric on $S$, the Cauchy sequence $(x_n)_{n \in \Z_{\geq 0}}$ in $T$ can be realized as points of a curve $C_{\widetilde{x}}$ in $T$ with bounded length. 
Denote by $C$ the connected component of $A$ containing $x_\infty$. 
Then there are finitely many loops $\gamma_1, \gamma_2, \ldots , \gamma_k \subset T$ which bound pairwise disjoint closed disks $D_1, D_2, \ldots , D_k$ in $T_{\mathrm{end}}$ with $\mathcal{E}_C(T) \subset \bigsqcup_{i=1}^k D_i$. 
Then there is a small positive number $\varepsilon > 0$ with $B_\varepsilon (\gamma_i) \subset T$ such that $B_\varepsilon (\gamma_i) \cap B_\varepsilon (\gamma_j) = \emptyset$ for any $i \neq j$, where $B_\varepsilon (\gamma)$ is the $\varepsilon$ \nbd of a subset $\gamma$ in $S$. 
Therefore $C_{\widetilde{x}}$ can intersect $\bigsqcup_i \gamma_i$ at most finitely many times. 
Then there are natural numbers $N$ and $l$ such that $\bigcup_{n > N} x_n \subset D_l \cap T$.  
This completes the claim. 
%means that an extension $q_{\mathrm{cpl}} \colon T_{\mathrm{cpl}} \to T_{\mathrm{end}}$ of the inclusion $T \subset T'$ is well-defined and continuous.  

We claim that, for any distinct connected components $A_1$ and $A_2$ of $A$, the image $q_{\mathrm{cpl}}(p_{\mathrm{cpl}}^{-1}(A_1)) \cap q_{\mathrm{cpl}}(p_{\mathrm{cpl}}^{-1}(A_2)) = \emptyset$. 
Indeed, fix distinct connected components $A_1$ and $A_2$ of $A$. 
By Hausdorff-Alexandroff theorem, since $A$ is compact metrizable, there is a continuous surjection $f \colon \mathcal{C} \to A$ from a Cantor set $\mathcal{C}$. 
Then the inverse images $\mathcal{C}_1 := f^{-1}(A_1)$ and $\mathcal{C}_2 := f^{-1}(A_2)$ are disjoint compact subsets of $\mathcal{C}$. 
Since the Cantor set $\mathcal{C}$ has a basis of closed and open subsets, the compactness of $\mathcal{C}_1$ implies that there is closed and open \nbd $V_1$ in $\mathcal{C}$ of $\mathcal{C}_1$ with $V_1 \cap \mathcal{C}_2 = \emptyset$. 
Then the complement $\mathcal{C} - V_1$ is a closed and open \nbd of $\mathcal{C}_2$. 
The images $F_1 := f(V_1)$ and $F_2 := f(\mathcal{C} - V_1)$ of compact subsets are disjoint compact subsets of $S$ with $A_1 \subseteq F_1$ and $A_2 \subseteq F_2$. 
The normality of $S$ implies that there are disjoint open \nbds $U_1$ and $U_2$ of closed subsets $F_1$ and $F_2$ respectively. 
Since $T$ is a connected component of $S - A$, by $A \subset U_1 \sqcup U_2$, the set difference $K := T \setminus (U_1 \sqcup U_2) = (T \sqcup A) \setminus (U_1 \sqcup U_2)$ is closed and so compact in $S$. 
This means that any end corresponding to a sequence of open subsets in $q_{\mathrm{cpl}}(p_{\mathrm{cpl}}^{-1}(A_1))$ is different from any end corresponding to a sequence of open subsets in  $q_{\mathrm{cpl}}(p_{\mathrm{cpl}}^{-1}(A_2))$.

The previous claim means that the image $p_{\mathrm{cpl}}(q_{\mathrm{cpl}}^{-1}(e))$ for any end $e$ is contained in a connected component of $A$. 
This implies that the image $p_0 \circ p_{\mathrm{cpl}}(q_{\mathrm{cpl}}^{-1}(e))$ is a point in $T'$, where $p_0 \colon T \sqcup A \to T'$ is the quotient map. 
Therefore a continuous extension $p \colon T_{\mathrm{end}} \to T'$ of the inclusion $T \subset T'$ is defined by $p(e) := p_0 \circ p_{\mathrm{cpl}}(q_{\mathrm{cpl}}^{-1}(e)) \in T'$ for any end $e$ as Figure~\ref{Fig:quotients}. 
Moreover, since $\mathcal{E}_C(T)$ is finite for any connected component $C$ of $A$ intersecting $\partial T$, any inverse images of singletons on $T'$ by $p$ are finite. 
\end{proof}

The non-existence of quasi-Q-sets consisting of singular points and non-recurrent points follows from the time-reversal symmetric condition for limit sets. 

\begin{lemma}\label{lem:03}
A flow with time-reversal symmetric limit sets on a compact connected surface has no quasi-Q-sets that consist of singular points and non-recurrent points. 
\end{lemma}

\begin{proof}
Let $v$ be a flow with time-reversal symmetric limit sets on a compact connected surface $S$. 
%By compactness, the complement of any disjoint union of infinitely many loops are disconnected. 
%
Assume that there is a quasi-Q-set $Q$ that consists of singular points and non-recurrent points. 
Fix a point $x$ whose $\omega$-limit and $\alpha$-limit sets are $Q$. 
%Since $Q$ consists of singular points and non-recurrent points, the point $x$ is non-recurrent and so $Q = \omega(x) = \alpha(x) = \overline{O(x)} - O(x)$. 
Since any connected component of a closed subset is closed, the intersection $\Sv \cap Q$ and its connected components are closed and so compact. 
Denote by $T$ the connected component of the complement $S - (\Sv \cap Q)$ containing $x$. 
Then $T$ is an invariant open subset of $S$ such that $\partial T \subseteq \Sv \cap Q$. 
%By definition of quasi-Q-set, there is a closed transversal $\mu \subset T$ intersecting $Q$ infinitely many times. 
%Therefore $T$ is essential. 
%
By \cite[Theorem 3]{richards1963classification}, the open surface $T$ is homeomorphic to the resulting surface from a compact surface by removing a closed totally disconnected subset and so the end completion $T_{\mathrm{end}}$ of $T$ is a compact surface. 
%Denote by $r_0 \colon T_{\mathrm{end}} \to T_0$ a homeomorphism. 
Let $v_{\mathrm{end}}$ be the resulting flow on $T_{\mathrm{end}}$ from $v$ by adding singular points. 
%
%Denote by $T_{\mathrm{cpl}}$ the metric completion of $T$ and by $p_{\mathrm{cpl}} \colon T_{\mathrm{cpl}} \to \overline{T} \subseteq T \sqcup (\Sv \cap Q)$ the canonical quotient map as in \S~\ref{sec:mc}. 
Collapsing any connected component of $\Sv \cap Q$ into singletons, we obtain the resulting space $T'$ from $T \sqcup (\Sv \cap Q)$. 
Let $p_0 \colon T \sqcup (\Sv \cap Q) \to T'$ be the quotient map and 
%Then $T$ is a proper subset of both $T_{\mathrm{end}}$ and $T'$. 
%1
$\mathcal{E}(T)$ the set of ends of $T$. 
Applying Lemma~\ref{lem:end_collapse} to $A = \Sv \cap Q$,  there is a continuous extension $p \colon T_{\mathrm{end}} \to T'$ of the inclusion $T \subset T'$ such that any inverse images of singletons on $T'$ by $p$ are finite. 
Thus the compact metrizable space $T'$ can be obtained from $T_{\mathrm{end}}$ by identifying finitely many ends finitely many times. 
Set $\mathcal{E}_{\mathrm{id}} := \{ e \in \mathcal{E}(T) \mid \{e \} \neq p^{-1}(p(e)) \}$. 
%Let $\mathcal{E} \subset \mathop{\mathrm{Sing}}(v_{\mathrm{end}})$ be the finite union of ends in $T_{\mathrm{end}}$ each of whose inverse images of images by $p$ is not a singleton. 
By definition of quasi-Q-set, there is a closed transversal $\mu \subset T$ intersecting $Q$ infinitely many times. 

We claim that there are infinitely many pairwise disjoint non-recurrent orbits in $Q$ intersecting $\mu$. 
Indeed, by \cite[Theorem~A]{yokoyama2021poincare}, the closures in $T_{\mathrm{end}}$ of any non-recurrent orbits in $Q$ is a closed interval and so intersects $\mu$ at most finitely many times. 
Therefore any non-recurrent orbits in $Q$ intersect $\mu$ at most finitely many times. 
This means that there are infinitely many pairwise disjoint non-recurrent orbits in $Q$ intersecting $\mu$. 

From \cite[Theorem~A]{yokoyama2021poincare}, the $\omega$-limit and $\alpha$-limit sets of non-recurrent orbits in $Q$ are contained in $\Sv \cap Q$. 
Then there are infinitely many pairwise disjoint non-recurrent orbits $O_i$ in $Q \subset T$ intersecting $\mu$. 
%By the finiteness of $p^{-1}$, 
Then the inverse images $p^{-1}(p_0(\omega(O_i))) \subset T_{\mathrm{end}}$ are singletons which are singular points for any $i$. 
%For any $i$, the inverse image $p^{-1}(p_0(\omega(O_i))) \subset T_{\mathrm{end}}$ is a singular point. 
Since $T'$ is obtained from $T_{\mathrm{end}}$ by identifying finitely many ends finitely many times, there are infinitely many loops $\nu_i \subset T_{\mathrm{end}}$ such that either $\nu_i = \{ \omega_{v_{\mathrm{end}}}(O_{i_1}) \} \sqcup O_{i_1}$ or $\nu_i = \{ \omega_{v_{\mathrm{end}}}(O_{i_1}), \alpha_{v_{\mathrm{end}}}(O_{i_1}) \} \sqcup O_{i_1} \sqcup O_{i_2}$ for some $i_1, i_2$, and that any pair of distinct loops $\nu_i$ and $\nu_j$ intersects at most two points.
Since any loop $\nu_i$ intersects the closed transversal $\mu$, the loop $p(\nu_i)$ is essential in $T'$. 
Therefore any loop $\nu_i$ either is essential or bounds an open disk containing ends in $\mathcal{E}_{\mathrm{id}}$. 

We claim that there are not infinitely many pairwise disjoint loops in $\bigcup_i \nu_i$. 
Indeed, assume that there are infinitely many pairwise disjoint loops in $\bigcup_i \nu_i$. 
By renumbering $O_i$, we may assume that loops $\nu_i$ are pairwise disjoint. 
From the finite existence of genus and boundary components, there are infinitely many loops such that any pairs of distinct such loops bound invariant open annuli. 
Therefore there are a closed annulus $A$ whose boundary consists of two distinct loops $\nu'$ and $\nu''$ in $\bigcup_i \nu_i$ and a loop $\nu'''$ in $\bigcup_i \nu_i$ such that $A - \nu'''$ is a disjoint union of two invariant annuli $A' \sqcup A''$. 
By $\nu' \sqcup \nu'' \sqcup \nu''' \subset p^{-1}(p(Q)) = \omega_{v_{\mathrm{end}}}(x)$, we have $x \in A' \cap A'' = \emptyset$, which contradict $x \notin \emptyset$. 

Thus we may assume that $\nu_i \cap \nu_j \neq \emptyset$ for any $i \neq j$. 
We claim that there are no  infinitely many loops $\nu_i$ such that $\nu_i = \{ \omega_{v_{\mathrm{end}}}(O_{i_1}) \} \sqcup O_{i_1} \subset T_{\mathrm{end}}$ for some $i_1$. 
Indeed, assume that there are infinitely many loops $\nu_i$ such that $\nu_i = \{ \omega_{v_{\mathrm{end}}}(O_{i_1}) \} \sqcup O_{i_1}$ for some $i_1$. 
Since $v$ has time-reversal symmetric limit sets, by renumbering $O_i$, we may assume that there is a singular point $s$ such that $\nu_i = \{ s \} \sqcup O_{i}$ for any $i$. 
Then $\bigcup_{i} \nu_i = \{ s \} \sqcup \bigsqcup_{i} O_i$.   
From the finite existence of genus and boundary components, there are infinitely many loops such that any pairs of distinct orbits $O_i$ and $O_j$ equipped with $s$ bound invariant open disks. 
Therefore there are a closed disk $B$ whose boundary is $\nu' \sqcup \{ s \} \sqcup \nu''$ in $\bigcup_{i} \nu_i$ and an orbit $\nu'''$ in $\bigcup_{i} \nu_i$ such that $B - \nu'''$ is a disjoint union of two invariant disks $B' \sqcup B''$. 
By $\nu' \sqcup \nu'' \sqcup \nu''' \subset p^{-1}(p(Q)) = \omega_{v_{\mathrm{end}}}(x)$, we have $x \in B' \cap B'' = \emptyset$, which contradict $x \notin \emptyset$. 

Thus there are infinitely many loops $\nu_i$ such that $\nu_i = \{ \omega_{v_{\mathrm{end}}}(O_{i_1}), \alpha_{v_{\mathrm{end}}}(O_{i_1}) \} \sqcup O_{i_1} \sqcup O_{i_2}$ for some $i_1, i_2$. 
Since $T'$ is obtained from $T_{\mathrm{end}}$ by identifying finitely many ends finitely many times, we have that $p(\omega_{v_{\mathrm{end}}}(O_{i_1}))  = p(\alpha_{v_{\mathrm{end}}}(O_{i_1})) \in p(\mathcal{E}_{\mathrm{id}})$ for any $i$, and that $\bigcup_i \nu_i \cap \mathop{\mathrm{Sing}}(v_{\mathrm{end}}) = \bigcup_i \{ \omega_{v_{\mathrm{end}}}(O_{i_1}), \alpha_{v_{\mathrm{end}}}(O_{i_1}) \} \subseteq \mathcal{E}_{\mathrm{id}}$ are finite. 
%By renumbering of $O_i$, we may assume that $\nu_i = \{ \omega_{v_{\mathrm{end}}}(O_{2i}), \alpha_{v_{\mathrm{end}}}(O_{2i}) \} \sqcup O_{2i} \sqcup O_{2i-1} \subset T_{\mathrm{end}}$ for any $i$. 
%Since $T'$ is obtained from $T_{\mathrm{end}}$ by identifying finitely many ends finitely many times, the time-reversal symmetry of limit sets for $v$ implies that $\bigcup_i \nu_i \cap \mathop{\mathrm{Sing}}(v_{\mathrm{end}}) = \bigcup_i \{ \omega_{v_{\mathrm{end}}}(O_{i_1}), \alpha_{v_{\mathrm{end}}}(O_{i_1}) \}$ are finite. 
By renumbering $O_i$, there are singular points $\omega$ and $\alpha$ such that $\nu_i = \{ \omega, \alpha \} \sqcup O_{2i} \sqcup O_{2i-1} \subset T_{\mathrm{end}}$ for any $i$. 
Since any loop $\nu_i$ either is essential or bound an open disk containing ends in $\mathcal{E}_{\mathrm{id}}$, by renumbering $O_i$, we may assume that two loops in $\{ \nu_i \}$ are homotopic relative to $\{ \omega, \alpha \}$ in $T_{\mathrm{end}}$ and so that any two orbits in $\{ O_i \}$ are homotopic relative to $\{ \omega, \alpha \}$ in $T_{\mathrm{end}}$. 
%Then $\bigcup_{i} \nu_i = \{ \omega, \alpha \} \sqcup \bigsqcup_{i} O_i$.   
Therefore there are pairwise distinct integers $i, j, k$ and a closed disk $B$ whose boundary is $O_i \sqcup \{ \omega, \alpha \} \sqcup O_j$ such that $B - O_k$ is a disjoint union of two invariant disks $B' \sqcup B''$. 
By $O_i \sqcup O_j \sqcup O_k \subset p_0^{-1}(p(Q)) \setminus \mathop{\mathrm{Sing}}(v_{\mathrm{end}}) \subset \omega_{v_{\mathrm{end}}}(x) \cup \alpha_{v_{\mathrm{end}}}(x)$, we have $x \in B' \cap B'' = \emptyset$, which contradict $x \notin \emptyset$.

%Thus non-recurrent orbits $O_i$ are not pairwise disjoint, which contradicts the choice of $O_i$. 
Thus there are no quasi-Q-sets that consist of singular points and non-recurrent points.
\end{proof}

\subsection{Proof of Theorem~\ref{main:01}}

Let $v$ be a flow on a compact connected surface $S$. 
Suppose that $v$ has time-reversal symmetric limit sets. 
From Lemma~\ref{lem:01} it follows that the existence of non-closed recurrent points implies that the flow $v$ is either an irrational flow or a Denjoy flow. 
Thus we may assume that any recurrent orbits are closed. 
Then $S = \Cv \sqcup \mathrm{P}(v)$. 
If there is a limit cycle, then Lemma~\ref{lem:01} implies that the surface is either a torus or a Klein bottle and the complement of the limit cycle is a transverse annulus.
Thus we may assume that there are no limit cycles. 
By \cite[Lemma 3.4]{yokoyama2017decompositions}, the non-existence of limit cycles implies that $\Pv$ is open. 
Then $\overline{\mathrm{P}(v)} \subseteq \Sv \sqcup \mathrm{P}(v)$. 
Lemma~\ref{lem:03} implies the non-existence of quasi-Q-sets that consist of singular points and non-recurrent points. 
\cite[Theorem~A]{yokoyama2021poincare} implies that any $\omega$-limit and $\alpha$-limit sets of non-recurrent points are either nowhere dense subsets in $\partial \Sv$ or  limit quasi-circuits.
For any point $x$ whose $\omega$-limit set is a  limit quasi-circuit, Lemma~\ref{lem:01} implies that the connected component of $W^u(\gamma) - \gamma = W^s(\gamma) - \gamma$ containing $x$ is an invariant open transverse annulus. 
Each non-recurrent orbit is either a connecting quasi-separatrix with time-reversal symmetric limit sets or is contained in an invariant open transverse annulus with time-reversal symmetric limit sets which are  limit quasi-circuits. 

Conversely, if $v$ is either an irrational flow or a Denjoy flow, then $v$ has time-reversal symmetric limit sets. 
If the surface is either a torus or a Klein bottle and there is a limit cycle whose complement is a transverse annulus, then $v$ has time-reversal symmetric limit sets. 
Suppose that each recurrent orbit is closed and each non-recurrent orbit is either a connecting quasi-separatrix with time-reversal symmetric limit sets or is contained in an invariant open transverse annulus with time-reversal symmetric limit sets which are  limit quasi-circuits. 
Then $S = \Cv \sqcup \mathrm{P}(v)$ and each non-recurrent point has time-reversal symmetric limit sets. 
This means that $v$ has time-reversal symmetric limit sets.

\section{Examples and proof of Corollary~\ref{main:02}}

By definition of time-reversal symmetric limit sets, we have the following examples. 

\begin{example}\label{ex:min_pp}
{\rm
Any minimal flow and any pointwise periodic flow have time-reversal symmetric limit sets.
%Any minimal flows and pointwise periodic flows have time-reversal symmetric limit sets. 
}
\end{example}

A generic Hamiltonian flow on a closed surface has time-reversal symmetric limit sets.
%Any generic Hamiltonian flows on closed surfaces have time-reversal symmetric limit sets. 
More precisely, we have the following examples. 

\begin{example}\label{ex:ham}
{\rm
Structurally stable Hamiltonian $C^r$-vector fields {\rm(}$r \in \Z_{>0}${\rm)} on closed surfaces form an open dense subset in the set of Hamiltonian $C^r$-vector fields and have time-reversal symmetric limit sets \cite[Theorem~2.3.8, p. 74]{ma2005geometric}. 
}
\end{example}

By surgeries on local structures, we can obtain following flows with Wada-Lakes-like structures. 

\begin{example}\label{ex:wada}
{\rm
There is a flow with time-reversal symmetric limit sets on a torus, for which there is the $\omega$-limit set of a point whose complement consists of an invariant open periodic center disk and an invariant open transverse annulus. 
Indeed, let $v_0$ be a flow with one circle $C = \mathop{\mathrm{Sing}}(v_0)$ and the complement is an invariant open transverse annuls as on the upper left figure of Figure~\ref{wada02}. 
\begin{figure}
\begin{center}
\includegraphics[scale=0.2]{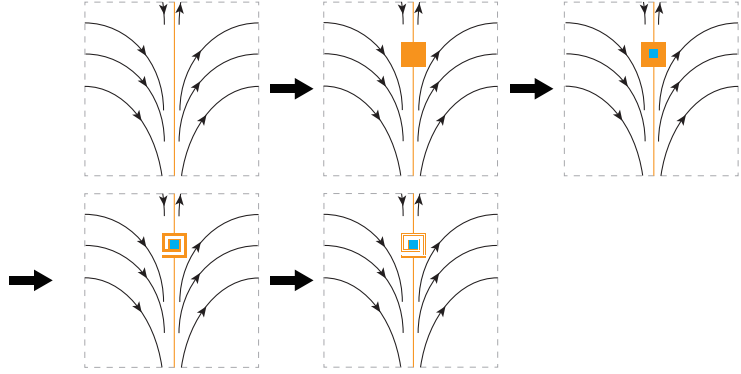}
\end{center}
\caption{A construction of a flow with time-reversal symmetric limit sets, for which there is the $\omega$-limit set of a point whose complement consists of an invariant open periodic center disk and an invariant open transverse annulus.}
\label{wada02}
\end{figure}
%(Notice that $v_0$ can be obtained from rational rotation on a torus by a turbulization along a closed transverse introduced by Reeb \cite{reeb1952certaines} (cf. \cite[Definition and Figure p.42]{Hector1983foliation_A}) and consider the resulting circle as singular points as on the upper middle of Figure~\ref{wada02}.)  
%Replace the limit cycle $C$ by a circle consisting of singular points, we obtain the resulting flow $v_0'$ as on the upper middle figure of Figure~\ref{wada02}. 
Replacing a singular point in $C$ into a rectangle $D$ consisting of singular points, we obtain the resulting flow $v_0'$ as on the upper middle figure of Figure~\ref{wada02}. 
Dividing the rectangle $D$ into 9 equal small rectangles and replacing the small middle rectangle $D_{22}$ into a center disk as on the middle left in Figure~\ref{wada01}, we obtain the resulting flow $v_1$ as on the upper right figure of Figure~\ref{wada02}. 
\begin{figure}
\begin{center}
\includegraphics[scale=0.045]{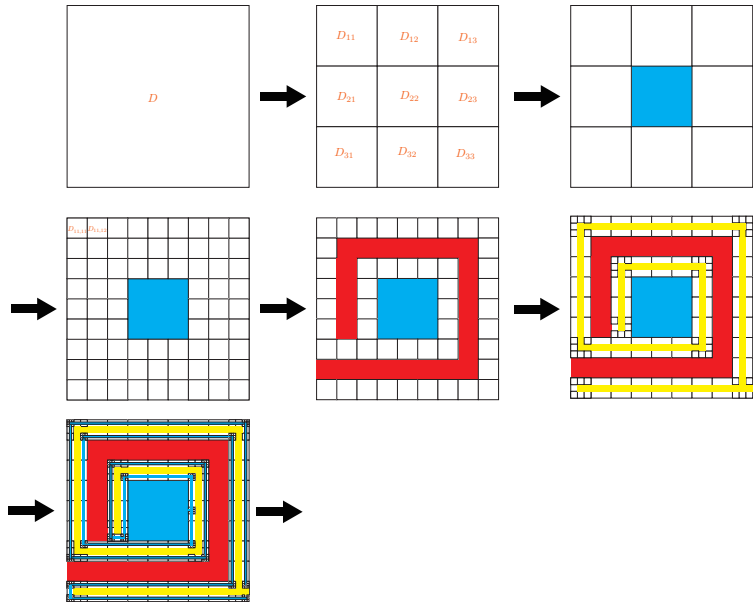}
\end{center}
\caption{A construction of a limit quasi-circuit whose complement consists of an invariant open periodic center disk and an invariant open transverse annulus by modifying the rectangle $D$.}
\label{wada01}
\end{figure}

As the construction of Lakes of Wada continua, digging a Lake as on the middle in Figure~\ref{wada01}, we obtain the resulting extended lake as on the lower left figure of Figure~\ref{wada02}. 
As the construction of Lakes of Wada continua, digging another lake as on the middle right in Figure~\ref{wada01}, we obtain the resulting extended lake as on the lower middle figure of Figure~\ref{wada02}. 
As the construction of Lakes of Wada continua, extending the center disk, and two Lakes infinitely many times, the resulting extended lakes $U_{1,\infty}$, $U_{2,\infty}$, and $U_{3,\infty}$ satisfy $\partial U_1 = \partial U_2 = \partial U_3$. 
Moreover, the three lakes $U_{1,\infty}$, $U_{2,\infty}$, and $U_{3,\infty}$ consists of rectangles as in Figure~\ref{wada01} and are formed as $U_i = \bigcup_{j = 0}^\infty U_{i,3j+i}$, where $U_{i,3j+i}$ is the extended area dug on the $3j+i_{\mathrm{th}}$ day. 
In addition, we can choose $U_{i,j}$ as $U_{i,3k+i} \cap \bigcup_{j<k} U_{i,3j+i} \subset U_{i,3k+i} \cap U_{i,3(k-1)+i}$ for any integer $k > 1$.
\begin{figure}
\begin{center}
\includegraphics[scale=0.75]{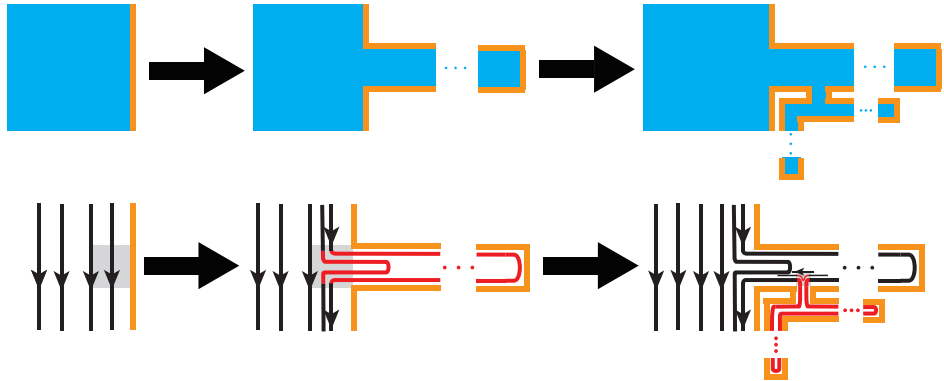}
\end{center}
\caption{The upper row represents deformations of the lake $\bigcup_{j = 0}^{k-1} U_{i,3j+i}$ on the $3k+i_{\mathrm{th}}$ day and the $3(k+1)+i_{\mathrm{th}}$ day, and the lower row represents associated deformations of flows along with extended areas dug on the $3k+i_{\mathrm{th}}$ day and the $3(k+1)+i_{\mathrm{th}}$ day.}
%Upper, deformations of the lake $\bigcup_{j = 0}^{k-1} U_{i,3j+i}$ on the $3k+i_{\mathrm{th}}$ day and the $3(k+1)+i_{\mathrm{th}}$ day; lower, associated deformations of flows along extended areas dug on the $3k+i_{\mathrm{th}}$ day and the $3(k+1)+i_{\mathrm{th}}$ day}
\label{wada_flow}
\end{figure}
In other words, when dig a Lake on the $3k+i_{\mathrm{th}}$ days, we replace the flow only on the union of the most recently dug area $U_{i,3(k-1)+i}$ and the new area $U_{i,3k+i}$ of the lake $\bigcup_{j = 0}^k U_{i,3j+i}$, by pinching an open trivial flow box in $U_{i,3(k-1)+i}$ and pulling orbit arcs in the flow box into $U_{i,3k+i}$ as on the left and middle of Figure~\ref{wada_flow}.

%As the construction of Lakes of Wada continua, d
Roughly speaking, digging a Lake as on the middle in Figure~\ref{wada01}, and pinching an open trivial flow box and pulling orbit arcs as on the middle in Figure~\ref{wada_flow}, we obtain the resulting flow $v_2$ as on the lower left figure of Figure~\ref{wada02}. 
More precisely, we construct such a flow by using seven kinds of insertions of continuous vector fields on squares with side length $3^{-n}$ for some $n \in \Z_{\geq 0}$ as in Figure~\ref{wada_flowbox01} and one replacement of an open trivial flow box as on the left of Figure~\ref{wada_flowbox01}. 
As the construction of Lakes of Wada continua, digging another Lake as on the middle right in Figure~\ref{wada01}, replacing one trivial flow box with side length $3^{-2}$ as in Figure~\ref{wada_flowbox01} and inserting seven kinds of squares with side length $3^{-2}$ as in Figure~\ref{wada_flowbox01}, we obtain the resulting flow $v_3$ as on the lower middle figure of Figure~\ref{wada02}. 
%As the construction of Lakes of Wada continua, extending the center disk, and two Lakes infinitely many times, the resulting flow $v_4$ is desired. 
By iterations of replacements and seven kinds of insertions infinitely many times, we can construct a flow with an invariant Wada-Lakes-like structure and make it smooth by Gutierrez's smoothing theorem~\cite{gutierrez1986smoothing} because any minimal sets are closed orbits. 

\begin{figure}
\begin{center}
\includegraphics[scale=0.75]{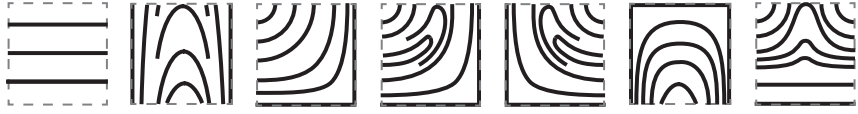}
\end{center}
\caption{Seven kinds of continuous vector fields on squares with side length $3^{-n}$ for some $n \in \Z_{\geq 0}$ to construct a flow with a Wada-Lakes-like structure.}
\label{wada_flowbox01}
\end{figure}
\begin{figure}
\begin{center}
\includegraphics[scale=0.65]{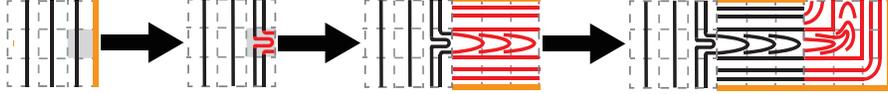}
\end{center}
\caption{The replacement of an open trivial flow box and examples of insertions of squares with side length $3^{-n}$ for some $n \in \Z_{\geq 0}$.}
\label{wada_flowbox02}
\end{figure}% 

In addition, increasing center disks and transverse annuli, we can realize arbitrarily finite numbers of connected components of the complement of the $\omega$-limit set of a point. 
Moreover, replacing center disks into periodic M\"obius bands, and cutting pairs of center disks and gluing the boundaries, the connected components can be not only center disks and open transverse annuli but also periodic M\"obius bands and periodic annuli. 
%
%Furthermore, if we replace $C - \{ x \}$ in $v_2$ with singular points then the resulting flow has time-reversal symmetric limit sets. 
}
\end{example}

\begin{example}\label{ex:wada04}
{\rm
There is a flow with time-reversal symmetric limit sets, for which there is a limit quasi-circuit whose complement consists of an invariant open periodic center disk and an invariant open transverse annulus on a closed orientable surface $\Sigma_2$ of genus two. 
Indeed, let $w_0$ be a flow on the torus $\T^2$ with one limit cycle and the complement is an invariant open transverse annuls as on the upper left figure of Figure~\ref{wada04}. 
Replace a point $x$ in the limit cycle $C$ by a closed interval $I_x$ consisting of singular points, we obtain the resulting flow $w_1$ with $\omega_{w_1}(C) = \alpha_{w_1}(C)= I_x = \mathop{\mathrm{Sing}}(w_1)$ such that the union $I_x \sqcup (C - \{ x \})$ is the $\omega$-limit and $\alpha$-limit set of any point in the complement $\T^2 - (I_x \sqcup (C - \{ x \}))$ as on the upper middle figure of Figure~\ref{wada04}. 
Replacing the two singular points on the boundary of the interval $I_x$ into closed disks $D_-$ and $D_+$ consisting of singular points respectively, we obtain the resulting flow $w_2$ as on the upper right figure of Figure~\ref{wada04}. 
\begin{figure}
\begin{center}
\includegraphics[scale=0.25]{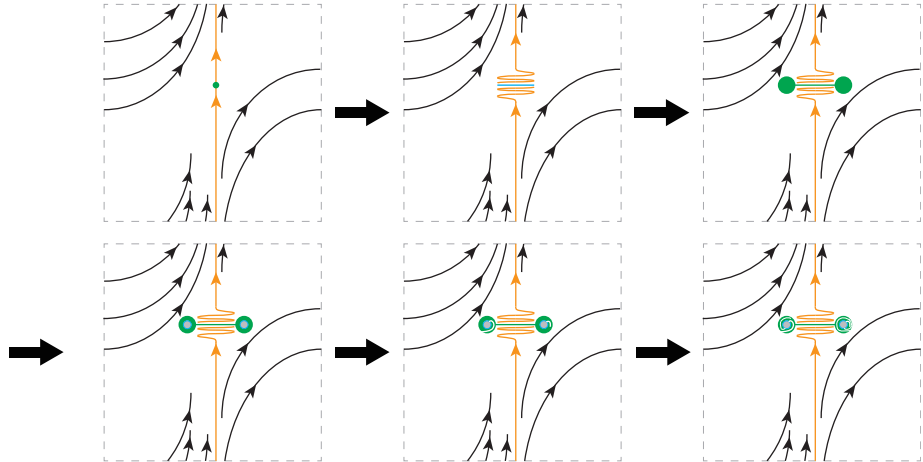}
\end{center}
\caption{Construction of a flow with time-reversal symmetric limit sets, for which there is a limit quasi-circuit whose complement consists of an invariant open transverse annulus.}
\label{wada04}
\end{figure}
Removing two invariant open disks $D_-'$ and $D_+'$ of half size in $D_-$ and $D_+$ as on the lower left figure of Figure~\ref{wada04} and pasting the new circular boundary components $\partial D_-'$ and $\partial D_+'$, we obtain that the resulting surface is an orientable closed surface $\Sigma_2$ of genus two and that the singular point set of the resulting flow $w_3$ on $\Sigma_2$ is the disjoint union of closed annulus $A$ and the open interval $\operatorname{int} I_x$. 
This closed annulus $(D_- - D_-') \cup (D_- - D_-')$ can be identified with the set difference $D - \mathrm{int} D_{22}$ as in Figure~\ref{wada01}. 
As the construction of the previous example, extending two Lakes infinitely many times, the resulting flow $w_4$ is desired. 

In addition, increasing center disks and transverse annuli, we can realize arbitrarily finite numbers of connected components of the complement of the quasi-circuit. 
Moreover, replacing center disks into periodic M\"obius bands, and cutting pairs of center disks and gluing the boundaries, the connected components can be not only center disks and open transverse annuli but also periodic M\"obius bands and periodic annuli. 
%
%Furthermore, if we replace $C - \{ x \}$ in $v_2$ with singular points then the resulting flow has time-reversal symmetric limit sets. 
}
\end{example}

\subsection{Proof of Corollary~\ref{main:02}}

From the constructions in Example~\ref{ex:wada}, consider a flow on a torus constructed in Figure~\ref{wada02} and as on the upper left figure of Figure~\ref{wada07}. 
\begin{figure}
\begin{center}
\includegraphics[scale=0.2]{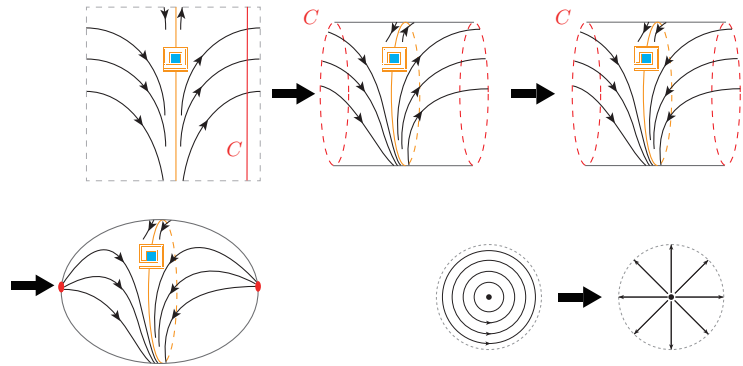}
\end{center}
\caption{Operations of cutting closed transversals, reversing the flow direction on the right part of a flow, and collapsing the new boundary components and a replacement of a center disk into a source disk.}
\label{wada07}
\end{figure}
Roughly speaking, by cutting closed transversals, reversing a part of a flow, and collapsing the new boundary components, we can construct flows on a sphere with Lakes of Wada attractors and with an arbitrarily large number of complementary domains as in Figure~\ref{wada07}. 

More precisely, consider a flow $v$ on a torus with time-reversal symmetric limit set, for which there is the $\omega$-limit set of a point whose complement consists of $N-1$ invariant open periodic center disks and an invariant open transverse annulus on a torus constructed in Example~\ref{ex:wada}. 
Replace any invariant open periodic center disks with invariant open source disks. 
Here a source disk is a flow on an open disk defined by a vector field $X(x,y) = (x,y)$ up to topological equivalence as on the lower right figure of Figure~\ref{wada07}. 
Take an essential closed transversal $C$ as on the upper left figure of Figure~\ref{wada07}. 
Remove $C$, we obtain the resulting open annulus $\T^2 - C$ and the flow $v'$ on it as on the upper middle figure of Figure~\ref{wada07}. 
To construct an attractor, reversing the flow direction on the right part of a flow as on the upper right figure of Figure~\ref{wada07}. 
Collapsing the new boundary components and adding two sources, we have the resulting sphere and the desired resulting flow as on the lower left figure of Figure~\ref{wada07}.

\bibliographystyle{abbrv}
\bibliography{../yt20211011}

\begin{thebibliography}{10}

\bibitem{aranson1996introduction}
S.~K. Aranson, G.~R. Beliski{\u\i}, and E.~Zhuzhoma.
\newblock {\em Introduction to the qualitative theory of dynamical systems on
  surfaces}.
\newblock American Mathematical Society, 1996.

\bibitem{athanassopoulos1992characterization}
K.~Athanassopoulos.
\newblock A characterization of denjoy flows.
\newblock {\em Bulletin of the London Mathematical Society}, 24(1):83--86,
  1992.

\bibitem{birkhoff1927dynamical}
G.~D. Birkhoff.
\newblock {\em Dynamical systems}, volume~9.
\newblock American Mathematical Soc., 1927.

\bibitem{boronski2020prime}
J.~P. Boro\'{n}ski, J.~\v{C}in\v{c}, and X.-C. Liu.
\newblock Prime ends dynamics in parametrised families of rotational
  attractors.
\newblock {\em J. Lond. Math. Soc. (2)}, 102(2):557--579, 2020.

\bibitem{brouwer1910analysis}
L.~E.~J. Brouwer.
\newblock Zur analysis situs.
\newblock {\em Mathematische Annalen}, 68(3):422--434, 1910.

\bibitem{cherry1937topological}
T.~Cherry.
\newblock Topological properties of the solutions of ordinary differential
  equations.
\newblock {\em American Journal of Mathematics}, 59(4):957--982, 1937.

\bibitem{coudene2006pictures}
Y.~Coudene.
\newblock Pictures of hyperbolic dynamical systems.
\newblock {\em Notices of the AMS}, 53(1):8--13, 2006.

\bibitem{fatou1920equations}
P.~Fatou.
\newblock Sur les {\'e}quations fonctionnelles.
\newblock {\em Bulletin de la soci{\'e}t{\'e} math{\'e}matique de France},
  48:208--314, 1920.

\bibitem{Freudenthal1931end}
H.~Freudenthal.
\newblock \"{U}ber die {E}nden topologischer {R}\"{a}ume und {G}ruppen.
\newblock {\em Math. Z.}, 33(1):692--713, 1931.

\bibitem{gutierrez1986smoothing}
C.~Guti{\'e}rrez.
\newblock Smoothing continuous flows on two-manifolds and recurrences.
\newblock {\em Ergodic Theory and dynamical systems}, 6(1):17--44, 1986.

\bibitem{Hector1983foliation_A}
G.~Hector and U.~Hirsch.
\newblock {\em Introduction to the geometry of foliations. {P}art {A}},
  volume~1 of {\em Aspects of Mathematics}.
\newblock Friedr. Vieweg \& Sohn, Braunschweig, 1981.
\newblock Foliations on compact surfaces, fundamentals for arbitrary
  codimension, and holonomy.

\bibitem{Hubbard1995henon}
J.~H. Hubbard and R.~W. Oberste-Vorth.
\newblock H\'{e}non mappings in the complex domain. {II}. {P}rojective and
  inductive limits of polynomials.
\newblock In {\em Real and complex dynamical systems ({H}iller\o d, 1993)},
  volume 464 of {\em NATO Adv. Sci. Inst. Ser. C: Math. Phys. Sci.}, pages
  89--132. Kluwer Acad. Publ., Dordrecht, 1995.

\bibitem{kelley1975general}
J.~L. Kelley.
\newblock General topology, gtm 27, 1975.

\bibitem{kennedy1991basins}
J.~Kennedy and J.~A. Yorke.
\newblock Basins of wada.
\newblock {\em Physica D: Nonlinear Phenomena}, 51(1-3):213--225, 1991.

\bibitem{ma2005geometric}
T.~Ma and S.~Wang.
\newblock {\em Geometric theory of incompressible flows with applications to
  fluid dynamics}.
\newblock Number 119. American Mathematical Soc., 2005.

\bibitem{marti2021wandering}
D.~Mart{\'\i}-Pete, L.~Rempe, and J.~Waterman.
\newblock Eremenko's conjecture, wandering lakes of wada, and maverick points.
\newblock {\em arXiv preprint arXiv:2108.10256}, 2021.

\bibitem{plykin1974sources}
R.~V. Plykin.
\newblock Sources and sinks of a-diffeomorphisms of surfaces.
\newblock {\em Mathematics of the USSR-Sbornik}, 23(2):233, 1974.

\bibitem{poincare1890}
H.~Poincar{\'e}.
\newblock Sur le probl{\`e}me des trois corps et les {\'e}quations de la
  dynamique.
\newblock {\em Acta Mathematica}, 13(1--2):1--270, 1890.

\bibitem{poincare1899}
H.~Poincar{\'e}.
\newblock {\em Les m{\`e}thodes nouvelles de la m{\'e}canique c{\'e}leste, Tome
  III}.
\newblock Gauthier-Villars (Paris), 1899.

\bibitem{raymond1960end}
F.~Raymond.
\newblock The end point compactification of manifolds.
\newblock {\em Pacific Journal of Mathematics}, 10(3):947--963, 1960.

\bibitem{richards1963classification}
I.~Richards.
\newblock On the classification of noncompact surfaces.
\newblock {\em Transactions of the American Mathematical Society},
  106(2):259--269, 1963.

\bibitem{yokoyama2017decompositions}
T.~Yokoyama.
\newblock Decompositions of surface flows.
\newblock {\em arXiv preprint arXiv:1703.05501}, 2017.

\bibitem{yokoyama2021poincare}
T.~Yokoyama.
\newblock {A Poincar{\'e}-Bendixson theorem for flows with arbitrarily many
  singular points}.
\newblock {\em arXiv preprint arXiv:2109.12478}, 2021.

\end{thebibliography}

\end{document}